\documentclass[12pt]{article}
\usepackage{enumerate}
\usepackage{amssymb,a4wide,latexsym,makeidx,fleqn}
\usepackage{amsthm}
\usepackage{amsmath}
\usepackage{graphicx}
\usepackage{float}
\usepackage{color}
\usepackage{indentfirst}
\usepackage{pdflscape}
\usepackage{siunitx}
\usepackage{diagbox}
\usepackage{rotating}
\usepackage{booktabs}
\usepackage{titling}
\usepackage[colorlinks, linkcolor=blue, anchorcolor=blue, citecolor=blue]{hyperref}
\usepackage[numbers,sort&compress]{natbib}

\allowdisplaybreaks[4]
\newtheorem{theorem}{Theorem}[section]

\newtheorem{lemma}[theorem]{Lemma}

\newtheorem{corollary}[theorem]{Corollary}

\newtheorem{claim}{Claim}
\newtheorem*{claim*}{Claim}

\begin{document}
	
	\textwidth 150mm \textheight 225mm
	\topmargin -1.4 cm  \parskip 4pt
	
	\title{Size and spectral conditions  for a graph with given minimum degree to be  $k$-$d$-critical\thanks{Supported by the
National Natural Science Foundation of China (Nos. 12271439 and 12371358)}}
	
	\author{{Zhenhao Zhang$^{a,b}$, Xiaogang Liu$^{a,b}$}, Ligong Wang$^{a,b,}$\thanks{Corresponding author.}\\
		{\small $^a$School of Mathematics and Statistics, Northwestern Polytechnical University,}\\ {\small  Xi'an, Shaanxi 710129,
			P.R. China.}\\
		{\small $^b$Xi'an-Budapest Joint Research Center for Combinatorics, Northwestern
			Polytechnical University,}\\
		{\small Xi'an, Shaanxi 710129,
			P.R. China. }\\
		{\small E-mail: zhangzhenhao@mail.nwpu.edu.cn, xiaogliu@nwpu.edu.cn, lgwang@nwpu.edu.cn} }
	\date{}
	\maketitle	
	\begin{center}
		\begin{minipage}{120mm}
			\vskip 0.3cm
			\begin{center}
				{\small {\bf Abstract}}
			\end{center}
			{\small   		
				A $k$-matching in a graph $G$ is defined as a function $f:E(G) \rightarrow \{0,1,\ldots,k\}$ satisfying $\sum_{e\in E_G(v)} f(e)$ $\leq k$ for each vertex $v\in V(G)$, where $E_G(v)$ denotes the set of edges incident to  $v$ in $G$. For $1\leq d\leq k$ and $d \equiv |V(G)|~(\mathrm{mod}~2)$, if for any $ v \in V(G)$,  there exists a $k$-matching $f$ such that $\sum_{e\in E_G(v)}f(e)=k-d$  and $\sum_{e\in E_G(u)}f(e)=k \text{ for every } u\in V(G)-\{v\}$, then $G$ is $k$-$d$-critical. A graph $G$ of odd order (resp. even order) is generalized factor-critical (resp. generalized bicritical)  if the empty set is the unique subset attaining the maximum value in the $k$--Berge--Tutte formula for $G$.
				In this paper, we provide sharp sufficient conditions in terms of size or spectral radius respectively for a graph $G$ to be $k$-$d$-critical, generalized factor-critical and generalized bicritical with given minimum
				degree.

				 }
		\end{minipage}
	\end{center}
	
	%%%%%%%%%%%%%%%%%%%%%%%%%%%%%%%%%%%%%%%%%%%%%%%%%%%%%%%%%%%%%%%%%%%%%%%%%%%%%%
	%%% Part 1: Introduction %%%
	\section{Introduction }
	\label{sec:ch6-introduction}
	% Para1: basic concept
	\noindent
	
	All graphs in this paper are simple, connected and undirected. Let $G$ be a graph with vertex set $V(G)$ and edge set $E(G)$. The order of $G$ is defined as the number of vertices of $G$, while its size, denoted by $e(G)$, refers to the number of edges of $G$.  For a vertex $v \in V(G)$, we write $E_G(v)$ for the set of edges incident with $v$, and $N_G(v)$ (or simply $N(v)$) for the neighborhood of $v$ in $G$. Let $\delta(G)$ denote the minimum degree of $G$.
	A vertex is called an isolated vertex in a graph $G$ if it is not adjacent to any other vertex in $G$. Denote by $I(G)$ the set of isolated vertices in $G$.
	Given two graphs $G$ and $H$, if $V(H)\subseteq V(G)$ and $E(H)\subseteq E(G)$, then $H$ is said to be a subgraph of $G$, denoted by $H\subseteq G$. If additionally $V(H)=V(G)$, then $H$ is referred to as a spanning subgraph of $G$, denoted by $H\preceq G$.
	For a subset $V_1 \subseteq V(G)$, the graph obtained by removing all vertices in $V_1$ along with their incident edges is denoted by $G-V_1$. When $V_1=\{v\}$, we simply write $G-v$.
	The disjoint union of two vertex-disjoint graphs $G_1$ and $G_2$, denoted by $G_1+G_2$, has vertex set $V(G_1)\cup V(G_2)$ and edge set $E(G_1)\cup E(G_2)$. The join of $G_1$ and $G_2$, written as $G_1\vee G_2$, is formed by taking their disjoint union and adding all edges between the two vertex sets.
	As usual, denote by $K_n$, $K_{a,n-a}$ and $C_n$ the complete graph, the complete bipartite graph and the cycle of order $n$, respectively. In particular, $K_0$  denotes the null graph.
	The complement of $G$, denoted by $\overline{G}$, shares the same vertex set as $G$, and two distinct vertices in $\overline{G}$ are adjacent if and only if they are not adjacent in $G$.
	For a graph $G$, let $\mbox{odd}(G)$ and $i(G)$ denote the number of nontrivial odd components and the number of isolated  vertices, respectively.	The adjacency matrix of a graph $G$ of order $n$ is denoted by $A(G)=(a_{i,j})_{n\times n}$, where $a_{i,j} = 1$ if vertices $v_i$ and $v_j$ are adjacent, and $a_{i,j} = 0$ otherwise.
	The adjacency spectral radius of $G$, denoted by $\rho(G)$, is defined as the largest eigenvalue of $A(G)$.
	
	A $k$-matching of a graph $G$ is defined as a function $f\colon E(G) \rightarrow \{0,1,\ldots,k\}$ satisfying $\sum_{e\in E_G(v)} f(e)$ $\leq k$ for every vertex $v\in V(G)$. The $k$-matching number of $G$, denoted by $\mu_k(G)$, is given by
	\[
	\mu_k(G)=\mbox{max}\left\{\sum_{e\in E(G)} f(e)~|~ \text{$f$ is a $k$-matching}\right\}.
    \]
	
	The $k$--deficiency of $G$ is $\mbox{def}_k(G) = k|V(G)| - 2\mu_k(G)$.
	Liu et al. \cite{Liu1,Liu2} gave the $k$--Berge--Tutte formula of a graph $G$:
	\[\mbox{def}_k(G)=\mathop{\mbox{max}}\limits_{S\subseteq V(G)} \begin{cases}
		k\cdot i(G-S)-k|S|, & \text{$k$ is even};     \\
		\mbox{odd}(G-S)+k\cdot i(G-S)-k|S|, &\text{$k$ is odd}.
	\end{cases}\]
	A  $k$--barrier of a graph $G$ is the subset $S\subseteq V(G)$ that attains the maximum value in the $k$--Berge--Tutte formula.
	Chang et al.  \cite{Chang} gave the definitions of a generalized factor-critical graph  (a $\text{GFC}_k$ graph for short) and a generalized bicritical graph (a $\text{GBC}_k$ graph for short). That is, a graph $G$ is called generalized factor-critical if $G$ has odd order and $\varnothing$ is its only $k$-barrier. Analogously, $G$ is called generalized bicritical if $G$ has even order and $\varnothing$ is its only $k$-barrier.

	For an odd integer $k\geq 3$, Liu et al. \cite{Liu2} established the following characterization of a $\text{GFC}_k$ graph.
	\begin{lemma}\emph{(\cite{Liu2})}
		If $G$ is a graph of odd order $n\geq 3$ and $k\geq 3$ is odd, then $G$ is $\text{GFC}_k$ if and only if for every $ v \in V(G)$,  there exists a $k$-matching $f$ such that
		\[ \sum_{e\in E_G(v)}f(e)=k-1 \text{~and~} \sum_{e\in E_G(u)}f(e)=k \text{~for~every~} u\in V(G)-\{v\}.
		\]
		
	\end{lemma}
	
	Chang et al.   \cite{Chang} extended the above result and defined a $k$-$d$-critical graph.  For an odd integer $k\geq 3$, let $1\leq d \leq k$ and $|V(G)|-d\equiv 0 ~(\mathrm{mod}~2)$. A graph $G$ is said to be $k$-$d$-critical if for every $v\in V(G)$, there exists a $k$-matching $f$ such that
	\[
	\sum_{e\in E_G(v)}f(e)=k-d
	\quad\text{and}\quad
	\sum_{e\in E_G(u)}f(e)=k
	\quad\text{for every } u\in V(G)\setminus\{v\}.
	\]
	Chang et al. \cite{Chang} also pointed out that for odd $k$,  a $k$-$d$-critical graph is $\text{GFC}_k$ when $d$ is odd, and $\text{GBC}_k$ when $d$ is even.

	A graph $G$ is said to have a $\{K_2,\{C_{2t+1}\colon t\geq 1\}\}$-factor if and only if $G$ contains a spanning subgraph each of whose components is either an odd cycle or a single edge. In 1953, Tutte  \cite{Tut} introduced the concept of a 2-matching and proved that the existence of a $\{K_2, \{C_{2t+1}\colon t\geq 1 \}\}$-factor in a graph is equivalent to the existence of a perfect 2-matching in a graph. Lu and Wang   \cite{Lu}  extended the notion of 2-matching to $k$-matching and characterized a necessary and sufficient condition for the existence of a $k$-matching in a graph.  Liu et al.  \cite{Liu1,Liu2} proposed the $k$--Berge--Tutte formula in a graph.
	Building on this foundation, Chang et al. \cite{Chang} extended the above results by introducing the notions of $\text{GFC}_k$ graphs, $\text{GBC}_k$ graphs and $k$-$d$-critical graphs, and established the corresponding necessary and sufficient conditions for their existence in a graph. A $k$-matching of a graph is the usual notion of a matching when $k=1$. A graph $G$ of order $n$ has a perfect $k$-matching if and only if $\mu_k(G)=\frac{nk}{2}$. O \cite{O} established a spectral condition for the existence of a perfect $k$-matching in a graph. For a graph $G$ with adjacency matrix $A(G)$ and diagonal degree matrix $D(G)$, define $A_\alpha(G)=\alpha D(G)+(1-\alpha)A(G), 0\leq\alpha\leq1.$	The $A_\alpha$-spectral radius of $G$ is the largest eigenvalue of $A_\alpha(G)$. Zhao et al. \cite{Zhao} later extended the result to the $A_\alpha$-spectral radius of a graph. Zhang and Lin \cite{Zhang1}  provided a distance spectral radius condition ensuring the existence of a perfect matching in a graph or a bipartite graph. Zhang et al. \cite{Zhang2} generalized
	and refined the previous result in \cite{Zhang1}. They investigated the existence of a
	perfect matching in a bipartite graph with prescribed minimum degree
	in terms of its distance spectral radius.
	Zhang and Fan   \cite{Zhang} proposed both a size condition and a spectral radius condition for the existence of a perfect $k$-matching in a graph. Niu et al.   \cite{Niu} further extended the spectral radius condition to the $A_\alpha$-spectral radius condition.
	The spectral radius of a graph is closely related to its fundamental properties.
	Therefore, it is a natural and well-motivated problem to characterize $k$-$d$-critical graphs, $\text{GFC}_k$ graphs and $\text{GBC}_k$ graphs in terms of either their size or their spectral radius.

	This paper is organized as follows. In Section 2, we introduce several lemmas that will be used in subsequent proofs. For odd $k\geq 3$, in Section 3, we establish tight sufficient conditions in terms of size or spectral radius respectively for a graph $G$ to be  $k$-$d$-critical.
	In Section 4, for even $k$, we establish  tight sufficient conditions in terms of size or spectral radius for a graph $G$ to be $\text{GFC}_k$ when its order is odd and $\text{GBC}_k$ when its order is even.
	
	\section{Preliminaries}\label{sec:Preliminaries}
	
	In this section, we  give some useful lemmas which will be used later.
	
	\begin{lemma}\label{lem1}\emph{(\cite{Chang})}
		Let $G$ be a graph of order $n\geq 3$ and $k\geq 2$.

		$\mathrm{(1)}$~When $k$ is even, $G$ is $\text{GFC}_k$ (resp.\ $\text{GBC}_k$) if and only if $n$ is odd (resp.\ even) and
		\[
		i(G - S) \leq |S| - 1  \text{ for any } \varnothing \neq S \subseteq V(G).
		\]
		
		$\mathrm{(2)}$~When $k$ is odd, $G$ is $\text{GFC}_k$ if and only if $n$ is odd and
		\[
		\mbox{odd}(G - S) + k \cdot i(G - S) \leq k|S| - 1  \text{ for any } \varnothing \neq S \subseteq V(G).
		\]
		
		$\mathrm{(3)}$~When $k$ is odd, $G$ is $\text{GBC}_k$ if and only if $n$ is even and
		\[
		\mbox{odd}(G - S) + k \cdot i(G - S) \leq k|S| - 2  \text{ for any } \varnothing \neq S \subseteq V(G).
		\]

	\end{lemma}
	
	\begin{lemma}\label{lem2}	\emph{(\cite{Chang})}
		Let $G$ be a  graph of order $n\geq 3$, let $k\geq 3$ be an odd integer, and let $1\leq d \leq k$ with  $n\equiv d~(\mathrm{mod}~2)$. $G$ is $k$-$d$-critical if and only if
		\[\mbox{odd}(G-S)+k\cdot i(G-S) \leq
		k|S|-d \text{ for any } \varnothing \neq S \subseteq V(G).
		\]
	\end{lemma}
	
	\begin{lemma}\label{lem3}	\emph{(\cite{Chang})}
		Let $G$ be a  graph of order $n\geq 3$, let $k\geq 3$ be an odd integer, and let $1\leq d \leq k$ with  $n\equiv d~(\mathrm{mod}~2)$. If $G$ is $k$-$d$-critical, then $G$ is $\text{GFC}_k$ when $d$ is odd and  $\text{GBC}_k$ when $d$ is even.
		
	\end{lemma}
	
	\begin{lemma}\label{lem4}	\emph{(\cite{Zheng})}
		Let $n = s + \sum_{i=1}^p n_i $. If $n_1 \geq n_2 \geq \cdots \geq n_p \geq t \geq 1 $ and $ n_1 < n - s - t(p-1) $, then
		\[
		e\bigl(K_s \vee (K_{n_1} + \cdots + K_{n_p})\bigr) < e\bigl(K_s \vee (K_{n-s-t(p-1)} + (p-1)K_t)\bigr).
		\]
	\end{lemma}
	
	\begin{lemma} \label{lem5} \emph{(\cite{Fan})}
		Let  $n =s+ \sum_{i=1}^p n_i$. If $n_1 \geq n_2 \geq \cdots \geq n_p \geq 1 $ and $ n_1 < n - s - p + 1 $, then
		\[
		\rho\bigl(K_s \vee (K_{n_1} + \cdots + K_{n_p})\bigr) < \rho\bigl(K_s \vee (K_{n-s-p+1} + (p-1)K_1)\bigr).
		\]
	\end{lemma}

	For a square matrix $M$, let $\eta(M)$ represent its spectral radius. Given two $n \times n$ matrices $A = (a_{i,j})$ and $B = (b_{i,j})$, define $A \geq B$ if $a_{i,j} \geq b_{i,j}$ for every pair of indices $i,j \in \{1,\dots,n\}$. $A = (a_{i,j} )$ is nonnegative if $a_{i,j} \geq 0$ for every pair of indices $i,j \in \{1,\dots,n\}$.

	\begin{lemma}\label{lem6}	\emph{(\cite{Ber})}
		Let $A = (a_{i,j} )$ and $B = (b_{i,j} )$ be two nonnegative square matrices of
		order $n$. If $A \geq  B$, then $\eta(A) \geq \eta(B)$. Furthermore, if $A$ is irreducible, $A \geq B$ and $A\not = B$, then $\eta(A)>\eta(B)$.
	\end{lemma}
	
	Suppose $M$ is a square matrix of order $n$, $\mathcal{N}=\{1, 2, \ldots, n\}$ and $\Pi:\mathcal{N}=\mathcal{N}_1\cup \mathcal{N}_2\cup \cdots \cup \mathcal{N}_k$ is a partition of $\mathcal{N}$.
	For each pair of indices $i,j \in \{1,2,\dots,k\}$, let $M_{i,j}$ be the submatrix of $M$ induced by rows $\mathcal{N}_i$ and columns $\mathcal{N}_j$. Then $M$ is similar to the matrix $M'$ below

	\[M'=\begin{pmatrix}
		M_{1, 1}&M_{1, 2}& \cdots & M_{1, k} \\
		M_{2, 1}& M_{2, 2}& \cdots &M_{2, k} \\
		\vdots&\vdots&\ddots&\vdots\\
		M_{k, 1}& M_{k, 2}& \cdots &M_{k, k} \\
	\end{pmatrix}.\]
	
	The quotient matrix of $M$ with respect to $\Pi$, written as $B_{\Pi}=(b_{i, j})$, is a $k\times k$ matrix whose entry $b_{i, j}$ represents the average row sum over the block $M_{i, j}$. The partition $\Pi$ is termed equitable if every block $M_{i, j}$ has constant row sum; in such a case, the quotient matrix $B_{\Pi}$ is referred to as equitable.
	
	%Also, we say that the quotient matrix $B_{\Pi}$ is equitable if $\Pi$ is an equitable partition of $M$.
	\begin{lemma}\label{lem7}	\emph{(\cite{Bro, God})}
		Let $M$ be a real symmetric matrix and $B_{\Pi}$ an equitable quotient matrix of $M$. Then the eigenvalues of $B_{\Pi}$ are also eigenvalues of $M$. Furthermore, if $M$ is nonnegative and irreducible, then the largest eigenvalue of $B_{\Pi}$ is also the largest eigenvalue of $M$.	
	\end{lemma}
	
	\begin{lemma}\label{lem10}(\cite{Hong})
		Let $G$ be a connected graph of order $n$ and size $m$. Then $\rho(G)\leq \sqrt{2m-n+1}$.
		Equality holds if and only if $G=K_n$ or $G=K_{1,n-1}$.
	\end{lemma}
	
	\begin{lemma}\label{lem8}
		For positive integers $n$, $\delta$ and $s$, let  $G_s=K_s \vee (K_{n-2s}+\overline{K_s})$ with $\delta \leq s \leq \lfloor \frac{n}{2} \rfloor$. The following hold:
		
		$\mathrm{(1)}$~$e(G_\delta)> e(G_s)$ with $\delta < s \leq \lfloor \frac{n}{2} \rfloor$ for $n>6\delta+2$ or $n=6\delta+1$.
		
		$\mathrm{(2)}$~$e(G_\delta)= e(G_{\lfloor \frac{n}{2} \rfloor})> e(G_s)$ with $\delta < s < \lfloor \frac{n}{2} \rfloor$ for $n=6\delta+2$ or $n=6\delta-1$.
		
		$\mathrm{(3)}$~$e(G_{\lfloor \frac{n}{2} \rfloor})> e(G_s)$ with $\delta \leq s < \lfloor \frac{n}{2} \rfloor$ for $n<6\delta-1$ or $n=6\delta$.
	\end{lemma}
	
	\begin{proof}
		Since $e(G_s)=\frac{3s^2}{2}+(\frac{1}{2}-n)s+\frac{n^2-n}{2}$, it is a quadratic
		function in the variable $s$ that opens upward and its axis of symmetry is at $s = \frac{2n-1}{6}$. Therefore, the maximum is attained at one or both endpoints
		of the interval.
		If $\frac{2n-1}{6}-\delta>\lfloor \frac{n}{2} \rfloor-\frac{2n-1}{6}$,  then $n>6\delta+2$ for even $n$ and $n>6\delta-1$ for odd $n$, and so $e(G_\delta)>e(G_{\lfloor \frac{n}{2} \rfloor})$. If $\frac{2n-1}{6}-\delta=\lfloor \frac{n}{2} \rfloor-\frac{2n-1}{6}$,  then $n=6\delta+2$ or $n=6\delta-1$, and so $e(G_\delta)=e(G_{\lfloor \frac{n}{2} \rfloor})$. If $\frac{2n-1}{6}-\delta<\lfloor \frac{n}{2} \rfloor-\frac{2n-1}{6}$,  then $n<6\delta+2$ for even $n$ and $n<6\delta-1$ for odd $n$, and so $e(G_\delta)<e(G_{\lfloor \frac{n}{2} \rfloor})$.
	\end{proof}
	
	\begin{lemma}\label{lem9}
		Let $n$, $\delta$ and $s$ be positive integers.  Denote $G_s=K_s \vee (K_{n-2s}+\overline{K_s})$ with $\delta \leq s \leq \lfloor \frac{n}{2} \rfloor$. If $n\geq 8\delta+4$, then $\rho(G_\delta)>\rho(G_s)$ with $\delta < s \leq \lfloor \frac{n}{2} \rfloor$.
	\end{lemma}
	\begin{proof}
		
		We  prove  the result  by distinguishing the following two cases.

		{\textbf{Case 1.}} $\delta < s < \frac{n}{2} $.
		
		For the partition $V(G_s) = V(K_{s}) \cup V(K_{n-2s})\cup V(\overline{K_{s}})$ of $G_s$, the quotient matrix of $A(G_s)$ is 	
		\[ M_s=\begin{pmatrix}
			s-1&n-2s&s\\
			s&n-2s-1&0\\
			s&0&0
		\end{pmatrix},\]
		and
		\[
		f_s(x)=|xI_3-M_s|=x^3+(s-n+2)x^2-(s^2-s+n-1)x+(n-1)s^2-2s^3.
		\]
		Similarly,
		\[
		f_\delta(x)=|xI_3-M_\delta|=x^3+(\delta-n+2)x^2-(\delta^2-\delta+n-1)x+(n-1)\delta^2-2\delta^3.
		\]
		Thus
		\[
		f_s(x)-f_\delta(x)=-(s-\delta)(-x^2+(\delta+s-1)x+\delta+s-\delta n+2\delta s-ns+2\delta^2+2s^2).
		\]
		Denote
		$$
		g(x)=-x^2+(\delta+s-1)x+\delta+s-\delta n+2\delta s-ns+2\delta^2+2s^2.
		$$
		Note that $\rho(K_\delta \vee (K_{n-2\delta}+\overline{K_\delta}))>\rho(K_{n-\delta})=n-\delta-1$.
		
		\begin{claim} \label{claim1}
			$g(x)<0$ for $x\in [n-\delta-1, +\infty)$.
		\end{claim}
		
		\noindent\emph{Proof of Claim 1.}~Since $g(x)$ is a quadratic function in the variable $x$ that opens downward and its axis of symmetry is at $x = \frac{\delta+s-1}{2}$, $n \geq 8\delta+4$ and $\delta < s <\frac{n}{2}$, we have $(n-\delta-1)-(\frac{\delta+s-1}{2}) >\frac{3}{4}n-\frac{3}{2}\delta-\frac{1}{2}>0$. Hence, it suffices to prove that $g(n-\delta-1)<0$.
		
		If $x=n-\delta-1$, then
		$$
		g(n-\delta-1)=2s^2+\delta s+\delta-\delta n-(\delta-1)(\delta-n+1)+2\delta^2-(\delta-n+1)^2.
		$$
		Let $h(s)=g(n-\delta-1)$.
		It is a quadratic function in the variable $s$ that opens upward and its axis of symmetry is at $s = -\frac{\delta}{4}$.  Then $h(s)$ is a monotonically increasing function on $\delta < s <\frac{n}{2}$.
		Recall that $n \geq 8\delta+4$. If $n$ is odd, then $h(\frac{n-1}{2})=-\frac{n^2}{2}+\frac{5\delta n}{2}-\frac{3\delta}{2}+\frac{1}{2}=\frac{n}{2}(5\delta-n)+\frac{1-3\delta}{2}<0$; if $n$ is even, then $h(\frac{n}{2}-1)=-\frac{n^2}{2}+\frac{(5\delta-2) n}{2}-2\delta+2=\frac{n}{2}(5\delta-n-2)+2(1-\delta)<0$. Hence $h(s)<0$ with $\delta < s <\frac{n}{2}$, that is, $g(n-\delta-1)<0$. Since  $g(x)$ is monotonically decreasing in $[n-\delta-1, +\infty)$, we have $g(x)<0$ for $x\in [n-\delta-1, +\infty)$.

		By Claim \ref{claim1},
		$f_s(x)-f_\delta(x)=-(s-\delta)g(x)>0$ for $x\in [n-\delta-1, +\infty)$. Denote by  $\theta_s$ and $\theta_\delta$ the largest eigenvalues of $G_s$ and $G_\delta$, respectively. By Lemma \ref{lem7}, $\theta_s$ is the largest root of $f_s(x) = 0$ and $\theta_\delta$ is the largest root of $f_\delta(x) = 0$. 	Note that $f_s(\theta_s)-f_\delta(\theta_s)=0-f_\delta(\theta_s)>0$, that is, $f_\delta(\theta_s)<0$.  Thus $\theta_s<\theta_\delta$, and then $\rho(G_s)<\rho(G_\delta)$.
		
		{\textbf{Case 2.}} $s= \frac{n}{2} $. Thus $n$ must be even.
		
		In this case, let $\widetilde{G}=G_{\frac{n}{2}}=K_\frac{n}{2} \vee \overline{K_\frac{n}{2}}$. For the partition $V(\widetilde{G}) = V(K_{\frac{n}{2}}) \cup V(\overline{K_\frac{n}{2}})$, the quotient matrix of $A(\widetilde{G})$ is 	
		\[ \widetilde{M}=\begin{pmatrix}
			\frac{n-2}{2}&\frac{n}{2}\\
			\frac{n}{2}&0
		\end{pmatrix},\]
		and
		$\widetilde{f}(x)=|xI_2-\widetilde{M}|=x^2-(\frac{n}{2}-1)x-\frac{n^2}{4}$.
		Let $\widetilde{\theta}$ be the largest root of $\widetilde{f}(x)=0$. Then $\widetilde{\theta}=\rho(\widetilde{G})$ by Lemma \ref{lem7}.
		Since $\rho(\widetilde{G})>\rho(K_\frac{n}{2} \vee \overline{K_1})=\frac{n}{2}$, the largest root of $\widetilde{f}(x)=0$ is the same as that of $(x-\frac{n}{2})\widetilde{f}(x)$.
		
		Let $$
		g(x)=f_\delta(x)-\left(x-\frac{n}{2}\right)\widetilde{f}(x) =(\delta+1)x^2+\left(-\delta^2+\delta-\frac{n}{2}+1\right)x-2\delta^3+\delta^2n -\delta^2-\frac{n^3}{8}.
		$$
		
		\begin{claim}
			$g(x)<0$ for $x\in (\frac{n}{2},n-1]$.
		\end{claim}
		
		\noindent\emph{Proof of Claim 2.}~Since $n\geq 8\delta+4$, we have
		\begin{align*}
			g\left(\frac{n}{2}\right)& =-\frac{n^3}{8}+\frac{\delta}{4}n^2+\frac{\delta^2+\delta+1}{2}n-2\delta^3-\delta^2\\
			&\leq -\frac{6\delta+4}{8}n^2+\frac{\delta^2+\delta+1}{2}n-2\delta^3-\delta^2\\
			&\leq -\frac{(11\delta^2+13\delta+3)}{2}n-2\delta^3-\delta^2\\
			&<0,
		\end{align*}	
		and
		\begin{align*}
			g(n-1)&=-\frac{n^3}{8}+(\delta+\frac{1}{2})n^2-(\delta+\frac{1}{2})n-2\delta^3\\
			%&\leq (-\frac{8\delta+4}{8}+\delta+\frac{1}{2})n^2-(\delta+\frac{1}{2})n-2\delta^3\\
			&\leq -(\delta+\frac{1}{2})n-2\delta^3\\&<0.
		\end{align*}

		Note that $g(x)$ is a quadratic function opening upward and hence is convex. We have already shown that $g(\frac{n}{2})<0$ and $g(n-1)<0$. Thus, $g(\lambda \frac{n}{2}+(1-\lambda)(n-1))\leq \lambda g(\frac{n}{2})+(1-\lambda)g(n-1)<0$ with $0\leq \lambda \leq 1$, and then $g(x)<0$ for $x\in (\frac{n}{2},n-1]$.
		
		Therefore, $0>g(\widetilde{\theta})= f_\delta(\widetilde{\theta})- (\widetilde{\theta}-\frac{n}{2})\widetilde{f}(\widetilde{\theta}) =f_\delta(\widetilde{\theta})$. This means that $\widetilde{\theta}<\theta_{\delta}$, and then $\rho(\widetilde{G})<\rho(G_\delta)$.
	\end{proof}
	
	\section{Size and spectral radius conditions for a graph to be $k$-$d$-critical with odd $k\geq 3$ involving minimum degree}
	
	For a graph $G$ of order $n$ with minimum degree $\delta(G)\geq \delta$, we establish  tight sufficient conditions in terms of size or spectral radius for $G$ to be $k$-$d$-critical, where $k\geq 3$ is an odd integer. Theorem  \ref{thm1} gives the size condition for $G$ to be $k$-$d$-critical. Theorem    \ref{thm2} gives the spectral radius condition for $G$ to be $k$-$d$-critical. Furthermore, by Lemma \ref{lem3}, we also establish  tight sufficient conditions in terms of size or spectral radius for $G$ to be $\text{GFC}_k$ or $\text{GBC}_k$.
	
	\noindent \textbf{Remark.}
	Observe that if $d=k$ and $n$ is even, then the complete graph $K_n$ is not $k$-$d$-critical.  As a result, it is impossible to establish an upper bound on either the size or the spectral radius beyond which every graph would necessarily be $k$-$d$-critical. In the following discussion, we only consider the case $1\leq d< k$.
     Let $o(G)$ denote the number of odd components of a graph $G$. When $k=1$, we have
		
	\[\mathrm{def}_{1}(G)=\max_{S\subseteq V(G)}
		\bigl\{1\cdot\mathrm{odd}(G-S)+1\cdot i(G-S)-|S|\bigr\}=\max_{S\subseteq V(G)}
		\bigl\{o(G-S)-|S|\bigr\}.\]

		The above notion reduces to the classical factor-criticality.  Spectral conditions for factor-critical graphs with given minimum degree have already been well studied. Zhang and Fan   \cite{Zhang} established sufficient conditions in terms of the size and the spectral radius for $k$-factor-critical graphs with given minimum degree. For this reason, in what follows we restrict our attention to odd integers $k\geq 3$.
	
	\begin{theorem}\label{thm1}
		For a positive integer $\delta$, let $G$ be a graph of order $n\geq 3$ with minimum degree $\delta(G) \geq \delta$, let $k\geq 3$ be an odd integer, and let $1\leq d < k$ with  $n\equiv d~(\mathrm{mod}~2)$.
		
		$\mathrm{(1)}$~If $n>6\delta+2$ or $n=6\delta+1$ and $e(G)\geq e(K_\delta \vee (K_{n-2\delta}+\overline{K_\delta}))$, then $G$ is $k$-$d$-critical unless $ G = K_\delta \vee (K_{n-2\delta}+\overline{K_\delta})$.
		
		$\mathrm{(2)}$~If $n=6\delta+2$ or $n=6\delta-1$ and $e(G)\geq e(K_\delta \vee (K_{n-2\delta}+\overline{K_\delta}))$, then $G$ is $k$-$d$-critical unless $ G = K_\delta \vee (K_{n-2\delta}+\overline{K_\delta})$ or $G =K_{\lfloor \frac{n}{2} \rfloor} \vee \overline{K_{\lceil \frac{n}{2} \rceil}} $.
		
		$\mathrm{(3)}$~If $n<6\delta-1$ or $n=6\delta$ and $e(G)\geq e(K_{\lfloor \frac{n}{2} \rfloor} \vee \overline{K_{ \lceil \frac{n}{2} \rceil}})$, then $G$ is $k$-$d$-critical unless $ G = K_{\lfloor \frac{n}{2} \rfloor} \vee \overline{K_{\lceil \frac{n}{2} \rceil}}$.
	\end{theorem}
	
	\begin{proof}
		By contradiction, assume that $G$ is not $k$-$d$-critical.  By Lemma \ref{lem2}, there exists $\varnothing \neq S \subseteq V(G)$ such that 
		
		\begin{equation}\label{1}
			\mbox{odd}(G-S)+k\cdot i(G-S)\geq k|S|-d+1. \tag{1}
		\end{equation}
		
		Let $|S|=s$ and $i(G-S)=i$. Suppose that the components of $G-S$ are $G^{1}, G^{2}, \ldots,G^{p}$ and $|V(G^1)|=n_1, |V(G^2)|=n_2, \ldots, |V(G^p)|=n_p$ with $n_1 \geq n_2 \geq \cdots \geq n_p \geq 1 $. Then $G$ is a spanning subgraph of $K_s \vee (K_{n_1}+K_{n_2}+\cdots+K_{n_p})$. 
		
		{\textbf{Case 1.}} $\mbox{odd}(G-S)\leq 1$.
		
		If $i\leq s-1$, by (\ref{1}), we have $ks-d+1\leq 1+ki\leq 1+k(s-1)$. Then $d\geq k$, a contradiction. Thus, $i\geq s$. We now proceed to discuss two subcases as follows.
		
		{\textbf{Subcase 1.}} There are nontrivial components in $G-S$. Since $i\geq s$ and $d\geq 1$, (\ref{1}) always holds. We can add all possible edges such that the nontrivial components form a clique in $G-S$. Then $G\preceq K_s \vee (K_{n-s-i}+\overline{K_i})$. If $i>s$, we can join the $i-s$ isolated vertices in $G-S$ one by one to all vertices in $V(K_{n-s-i})$ and $G\preceq K_s \vee (K_{n-s-i}+\overline{K_i}) \preceq K_s \vee (K_{n-2s}+\overline{K_s})$. Since $\delta(G)\geq \delta$ and $K_{n-2s}$ is a nontrivial component, $\delta \leq s \leq \frac{n-2}{2}$.
		
		{\textbf{Subcase 2.}} There exist no nontrivial components in $G-S$. In this case, $G\preceq K_s \vee \overline{K_{n-s}}\preceq K_{\lfloor \frac{n}{2} \rfloor} \vee \overline{K_{\lceil \frac{n}{2} \rceil}}$.

		Let $G_s=K_s \vee (K_{n-2s}+\overline{K_s})$. It is noted that $K_{\lfloor \frac{n}{2} \rfloor} \vee \overline{K_{\lceil \frac{n}{2} \rceil}}=K_{\frac{n-1}{2} } \vee \overline{K_{ \frac{n+1}{2} }}=K_{\frac{n-1}{2} } \vee (K_1+\overline{K_{ \frac{n-1}{2} }})=G_{\frac{n-1}{2}}$ with odd $n$; $K_{\lfloor \frac{n}{2} \rfloor} \vee \overline{K_{\lceil \frac{n}{2} \rceil}}=K_{\frac{n}{2} } \vee \overline{K_{ \frac{n}{2} }}=K_{\frac{n}{2} } \vee (K_0+\overline{K_{ \frac{n}{2} }})=G_{\frac{n}{2}}$ with even $n$. Hence $G\preceq G_s$ and $e(G)\leq e(G_s)$ for some $s\in [\delta,\lfloor \frac{n}{2} \rfloor]$.
		
	 {\textbf{Case 2.}} $\mbox{odd}(G-S)> 1$.
		
		In this case, we have $n_1<n-s-p+1$. By Lemma \ref{lem4}, $e(K_s \vee (K_{n_1}+K_{n_2}+\cdots+K_{n_p})) < e(K_s \vee (K_{n-s-p+1}+(p-1)K_1))$. Let $G'=K_s \vee (K_{n-s-p+1}+(p-1)K_1)$. Then $\mbox{odd}(G'-S)\leq1$, $e(G) < e(G')$ and $G'$ satisfies (\ref{1}) with respect to $S$. However, the assumption $\delta(G)\geq\delta$ does not necessarily imply that $\delta(G')\geq\delta$. If $s\geq\delta$, then $\delta(G')=s\geq\delta$, and we may apply the argument in Case 1 to $G'$. 
		
		If $s<\delta$, then every component of $G-S$ has order at least $t=\delta-s+1\geq2$. By (\ref{1}), $p\geq\mbox{odd}(G-S)\geq ks-d+1\geq k(s-1)+2\geq2s$. Hence $n\geq s+2st$. By Lemma \ref{lem4}, we have $e(G)\leq e\bigl(K_s\vee(K_{n-s-(2s-1)t}+(2s-1)K_t)\bigr)$. Let $G''=K_s\vee(K_{n-s-(2s-1)t}+(2s-1)K_t)$. Since $\delta=s+t-1$, 
		\[
		\begin{aligned}
			2\bigl(e(G_\delta)-e(G'')\bigr)
			={}&2(s-1)(2t-1)(n-s-2st)\\
			&+(s-1)^2(2t-1)^2+(s-1)(2t^2-1)+t(t-1)>0.
		\end{aligned}
		\]
		Therefore $e(G)\leq e(G'')<e(G_\delta)$. Moreover, $n\geq s+2st=2\delta+s+2(s-1)(t-1)>2\delta$. By Lemma \ref{lem8}, $e(G)<e(G_\delta)\leq\max\left\{e(G_\delta),e(G_{\lfloor n/2\rfloor})\right\}$, contrary to the corresponding size condition in Theorem \ref{thm1}.
		
		Now we need to compare  the number of edges of $G_s$ for $\delta \leq s \leq \lfloor \frac{n}{2} \rfloor$. By Lemma \ref{lem8}, when $n>6\delta+2$ or $n=6\delta+1$, $e(G)\leq e(G_s)\leq e(G_\delta)$  with equality  throughout if and only if $G=G_\delta$ and $s=\delta$. When $n=6\delta+2$ or $n=6\delta-1$, then $e(G)\leq e(G_s)\leq e(G_\delta)=e(G_{\lfloor \frac{n}{2} \rfloor})$ with equality  throughout if and only if either $G=G_\delta$ and $s=\delta$, or $G=G_{\lfloor n/2\rfloor}$ and
		$s=\lfloor n/2\rfloor$. When $n<6\delta-1$ or $n=6\delta$, $e(G)\leq e(G_s)\leq e(G_{\lfloor \frac{n}{2} \rfloor})$ with equality  throughout if and only if $G=G_{\lfloor \frac{n}{2} \rfloor}$ and $s=\lfloor \frac{n}{2} \rfloor$. 
        It leads to a contradiction with the conditions of Theorem \ref{thm1}.
		This completes the proof. 
		
	\end{proof}

	\begin{theorem}\label{thm2}
		For a positive integer $\delta$, let $G$ be a graph of order $n\geq 8\delta+4$ with minimum degree $\delta(G) \geq \delta$. Let $k\geq 3$ be odd and let $1\leq d < k$ with  $n\equiv d ~(\mathrm{mod}~2)$.	
		If  $\rho(G)\geq \rho(K_\delta \vee (K_{n-2\delta}+\overline{K_\delta}))$, then $G$ is $k$-$d$-critical unless $ G = K_\delta \vee (K_{n-2\delta}+\overline{K_\delta})$.
	\end{theorem}
	
	\begin{proof}
		By contradiction, assume that $G$ is not $k$-$d$-critical.
		By Lemma \ref{lem2}, there exists $\varnothing \neq S \subseteq V(G)$ such that \begin{equation}\label{2}
			\mbox{odd}(G-S)+k\cdot i(G-S)\geq k|S|-d+1. \tag{2}
		\end{equation}
		
		Let $|S|=s$ and $i(G-S)=i$. Suppose that the components of $G-S$ are $G^{1}$, $G^{2}, \ldots, G^{p}$ and $|V(G^1)|=n_1, |V(G^2)|=n_2,\ldots, |V(G^p)|=n_p$ with $n_1 \geq n_2 \geq \cdots \geq n_p \geq 1 $.  Then $G$ is a spanning subgraph of $K_s \vee (K_{n_1}+K_{n_2}+\cdots+K_{n_p})$. 
		
		{\textbf{Case 1.}} $\mbox{odd}(G-S)> 1$.
		 
		 In this case, we have $n_1<n-s-p+1$. By Lemma \ref{lem5}, $\rho(K_s \vee (K_{n_1}+K_{n_2}+\cdots+K_{n_p})) < \rho(K_s \vee (K_{n-s-p+1}+(p-1)K_1))$. Let $G'=K_s \vee (K_{n-s-p+1}+(p-1)K_1)$. Then $\mbox{odd}(G'-S)\leq1$, $\rho(G) < \rho(G')$ and $G'$ satisfies (\ref{2}) with respect to $S$. However, the assumption $\delta(G)\geq\delta$ does not necessarily imply that $\delta(G')\geq\delta$. If $s\geq\delta$, then $\delta(G')=s\geq\delta$, and we may apply the argument in Case 2 to $G'$. 
		 
		 If $s<\delta$, every isolated vertex of $G-S$, if any, has degree at most $s$ in $G$, contradicting $\delta(G)\geq\delta$. Thus $i=0$ and $n_p\geq 2$. Since $G-S$ has no isolated vertices, $\mbox{odd}(G-S)\equiv n-s\equiv ks-d\pmod2$. Hence $\mbox{odd}(G-S)\geq ks-d+1$ implies $\mbox{odd}(G-S)\geq ks-d+2$. By (\ref{2}), $\mbox{odd}(G-S)\geq ks-d+2\geq k(s-1)+3\geq3s$. In particular, $p\geq\mbox{odd}(G-S)\geq3s$.
		 
	     Put $a=\delta-s+1\geq2$ and $N=n-s$. For every $v\in V(G^j)$, $\delta\leq d_G(v)\leq s+n_j-1$, so $n_j\geq a$ for all $j$. Moreover, $(p-1)a\geq(3s-1)a=2\delta+(s-1)(3a-2)\geq2\delta$. Therefore, $n_j\leq N-(p-1)a\leq N-2\delta$ with $(1\leq j\leq p)$. Since $n\geq8\delta+4$ and $s\leq\delta-1$, we have
	     $N\geq7\delta+5>4\delta$. By convexity, using $\sum_{j=1}^{p}n_j=N$ and $n_j\leq N-2\delta$, we obtain
	     \[
	     \sum_{j=1}^{p}\binom{n_j}{2}
	     \leq
	     \binom{N-2\delta}{2}+\binom{2\delta}{2}.
	     \]
	     Hence
	     \begin{align*}
	     	e(G)
	     	&\leq\binom{s}{2}+sN+
	     	\sum_{j=1}^{p}\binom{n_j}{2}\\
	     	&\leq\binom{s}{2}+sN+
	     	\binom{N-2\delta}{2}+\binom{2\delta}{2}\\
	     	&=\binom n2-2\delta(n-s-2\delta).
	     \end{align*}
	     
	     By Lemma \ref{lem10},
	     \begin{align*}
	     	\rho(G)^2
	     	\leq2e(G)-n+1\leq(n-1)^2-4\delta(n-s-2\delta).
	     \end{align*}
	     
	     On the other hand,
	     \begin{align*}
	     	&(n-\delta-1)^2
	     	-\bigl((n-1)^2-4\delta(n-s-2\delta)\bigr)\\
	     	&\qquad=\delta(2n-4s-7\delta+2)\\
	     	&\qquad\geq\delta(5\delta+14)>0.
	     \end{align*}
	     
	     It follows that $\rho(G)<n-\delta-1$. Hence $\rho(G)<n-\delta-1=\rho(K_{n-\delta})<\rho(G_\delta)$, contrary to the assumed spectral condition.
	  
		{\textbf{Case 2.}} $\mbox{odd}(G-S)\leq 1$.
		 
		If $i\leq s-1$, by (\ref{2}), we have $ks-d+1\leq 1+ki\leq 1+k(s-1)$. Then $d\geq k$, a contradiction. Thus, $i\geq s$. 
		By Lemma \ref{lem6}, $\rho(K_s \vee (K_{n-s-i}+\overline{K_i}))\leq \rho(K_s \vee (K_{n-2s}+\overline{K_s}))$ and $\rho(K_s \vee \overline{K_{n-s}}) \leq \rho(K_{\lfloor \frac{n}{2} \rfloor} \vee \overline{K_{\lceil \frac{n}{2} \rceil}})$.
		
		Denote $G_s=K_s \vee (K_{n-2s}+\overline{K_s})$. Then $K_{\lfloor \frac{n}{2} \rfloor} \vee \overline{K_{\lceil \frac{n}{2} \rceil}}=G_{\lfloor \frac{n}{2} \rfloor}$. 
		Recall that $i\geq s$. Since $i=i(G-S)\leq n-s$, we have $s\leq i\leq n-s$, and hence $s\leq \left\lfloor\frac{n}{2}\right\rfloor$. Moreover, every isolated vertex of $G-S$ has degree at most $s$ in $G$. Therefore, $\delta(G)\leq s$. Since $\delta(G)\geq\delta$, it follows that $s\geq\delta$. Consequently, $\delta\leq s\leq \left\lfloor\frac{n}{2}\right\rfloor$. We now compare the spectral radii of $G_s$ over this range of $s$.

		By Lemma \ref{lem9}, if $G$ is not $k$-$d$-critical, either $G=K_\delta \vee (K_{n-2\delta}+\overline{K_\delta})$ or $\rho(G)< \rho(K_\delta \vee (K_{n-2\delta}+\overline{K_\delta}))$.  It leads to a contradiction with the conditions of Theorem \ref{thm2}. This completes the proof.
	\end{proof}
	By Lemma \ref{lem3}, if $G$ is $k$-$d$-critical, then $G$ is $\text{GFC}_k$ when $d$ is odd and $\text{GBC}_k$ when $d$ is even. 
	As direct consequences of Theorems \ref{thm1} and \ref{thm2}
	together with Lemma \ref{lem3}, we obtain the following corollaries.
	For graphs of odd order, we take $d$ to be odd, while for graphs of even
	order, we take $d$ to be even.

	\begin{corollary}
		For a positive integer $\delta$ and an odd $k\geq 3$, let $G$ be a graph of odd  order $n\geq 3$ with minimum degree $\delta(G) \geq \delta$.
		
		$\mathrm{(1)}$~If $n>6\delta-1$ and $e(G)\geq e(K_\delta \vee (K_{n-2\delta}+\overline{K_\delta}))$, then $G$ is $\text{GFC}_k$ unless $ G = K_\delta \vee (K_{n-2\delta}+\overline{K_\delta})$.
		
		$\mathrm{(2)}$~If $n=6\delta-1$ and $e(G)\geq e(K_\delta \vee (K_{n-2\delta}+\overline{K_\delta}))$, then $G$ is $\text{GFC}_k$ unless $ G = K_\delta \vee (K_{n-2\delta}+\overline{K_\delta})$ or $G =K_{\frac{n-1}{2} } \vee \overline{K_{ \frac{n+1}{2} }}$.
		
		$\mathrm{(3)}$~If $n<6\delta-1$ and $e(G)\geq e(K_{ \frac{n-1}{2} } \vee \overline{K_{  \frac{n+1}{2} }})$, then $G$ is $\text{GFC}_k$ unless $ G = K_{ \frac{n-1}{2} } \vee \overline{K_{ \frac{n+1}{2} }}$.
	\end{corollary}
	
	\begin{corollary}
		For a positive integer $\delta$ and an odd $k\geq 3$, let $G$ be a graph of even  order $n\geq 4$ with minimum degree $\delta(G) \geq \delta$.
		
		$\mathrm{(1)}$~If $n>6\delta+2$  and $e(G)\geq e(K_\delta \vee (K_{n-2\delta}+\overline{K_\delta}))$, then $G$ is $\text{GBC}_k$ unless $ G = K_\delta \vee (K_{n-2\delta}+\overline{K_\delta})$.
		
		$\mathrm{(2)}$~If $n=6\delta+2$  and $e(G)\geq e(K_\delta \vee (K_{n-2\delta}+\overline{K_\delta}))$, then $G$ is $\text{GBC}_k$ unless $ G = K_\delta \vee (K_{n-2\delta}+\overline{K_\delta})$ or $G =K_{ \frac{n}{2}} \vee \overline{K_{ \frac{n}{2} }} $.
		
		$\mathrm{(3)}$~If $n<6\delta+2$  and $e(G)\geq e(K_{ \frac{n}{2}} \vee \overline{K_{  \frac{n}{2} }})$, then $G$ is $\text{GBC}_k$ unless $ G = K_{ \frac{n}{2} } \vee \overline{K_{ \frac{n}{2} }}$.
	\end{corollary}
	
	\begin{corollary}
		For a positive integer $\delta$ and an odd $k\geq 3$, let $G$ be a graph of odd order $n\geq 8\delta+4$ with minimum degree $\delta(G) \geq \delta$.
		If  $\rho(G)\geq \rho(K_\delta \vee (K_{n-2\delta}+\overline{K_\delta}))$, then $G$ is $\text{GFC}_k$ unless $ G = K_\delta \vee (K_{n-2\delta}+\overline{K_\delta})$.
	\end{corollary}
	
	\begin{corollary}
		For a positive integer $\delta$ and an odd $k\geq 3$, let $G$ be a graph of even order $n\geq 8\delta+4$ with minimum degree $\delta(G) \geq \delta$.
		If  $\rho(G)\geq \rho(K_\delta \vee (K_{n-2\delta}+\overline{K_\delta}))$, then $G$ is $\text{GBC}_k$ unless $ G = K_\delta \vee (K_{n-2\delta}+\overline{K_\delta})$.
	\end{corollary}

	\section{Size and spectral radius conditions for a graph to be $\text{GFC}_k$ or $\text{GBC}_k$ with even $k$ involving minimum degree}
	
	For a graph $G$ of order $n$ satisfying $\delta(G)\geq \delta$, we establish  tight sufficient conditions for $G$ to be $\text{GFC}_k$ or $\text{GBC}_k$, where $k$ is an even integer, in terms of both size and spectral radius.
	
	\begin{theorem}\label{thm3}
		For a positive integer $\delta$ and an even $k$, let $G$ be a graph of odd  order $n\geq 3$ with minimum degree $\delta(G) \geq \delta$.
		
		$\mathrm{(1)}$~If $n>6\delta-1$ and $e(G)\geq e(K_\delta \vee (K_{n-2\delta}+\overline{K_\delta}))$, then $G$ is $\text{GFC}_k$ unless $ G = K_\delta \vee (K_{n-2\delta}+\overline{K_\delta})$.
		
		$\mathrm{(2)}$~If $n=6\delta-1$ and $e(G)\geq e(K_\delta \vee (K_{n-2\delta}+\overline{K_\delta}))$, then $G$ is $\text{GFC}_k$ unless $ G = K_\delta \vee (K_{n-2\delta}+\overline{K_\delta})$ or $G =K_{\frac{n-1}{2} } \vee \overline{K_{ \frac{n+1}{2} }} $.
		
		$\mathrm{(3)}$~If $n<6\delta-1$ and $e(G)\geq e(K_{ \frac{n-1}{2} } \vee \overline{K_{  \frac{n+1}{2} }})$, then $G$ is $\text{GFC}_k$ unless $ G = K_{ \frac{n-1}{2} } \vee \overline{K_{ \frac{n+1}{2} }}$.
	\end{theorem}
	
	\begin{proof}
		By contradiction, assume that $G$ is not $\text{GFC}_k$. 
		By Lemma \ref{lem1}(1), there exists $\varnothing \neq S \subseteq V(G)$ such that \begin{equation}\label{3}
			i(G-S)\geq |S|. \tag{3}
		\end{equation}
		
		Let $|S|=s$ and $i(G-S)=i$. If there are nontrivial components in $G-S$, then $G$ is a spanning subgraph of $K_s \vee (K_{n-s-i}+\overline{K_i})$. If $i>s$, we can join the $i-s$ isolated vertices in $G-S$ one by one to all vertices in $V(K_{n-s-i})$. Hence $G\preceq K_s \vee (K_{n-s-i}+\overline{K_i}) \preceq K_s \vee (K_{n-2s}+\overline{K_s})$. Since $\delta(G)\geq \delta$ and $K_{n-2s}$ is a nontrivial component, $\delta \leq s \leq \frac{n-2}{2}$. 
		If there is no nontrivial component in $G-S$, $G\preceq K_s \vee \overline{K_{n-s}}\preceq K_{ \frac{n-1}{2}} \vee \overline{K_{\frac{n+1}{2}}}$.
		Denote $G_s=K_s \vee (K_{n-2s}+\overline{K_s})$. Then $K_{\frac{n-1}{2}} \vee \overline{K_{\frac{n+1}{2}}}=G_{\frac{n-1}{2}}$. We now compare the size of $G_s$ for $\delta \leq s \leq \frac{n-1}{2}$.
		
		By Lemma \ref{lem8}, when $n>6\delta-1$, $e(G)\leq e(G_s)\leq e(G_\delta)$  with equality  throughout if and only if $G=G_\delta$ and $s=\delta$. When $n=6\delta-1$, then $e(G)\leq e(G_s)\leq e(G_\delta)=e(G_{\frac{n-1}{2}})$ with equality  throughout if and only if either $G=G_\delta$ and $s=\delta$, or $G=G_{\frac{n-1}{2}}$ and
		$s=\frac{n-1}{2}$. When $n<6\delta-1$, $e(G)\leq e(G_s)\leq e(G_{\frac{n-1}{2}})$ with equality  throughout if and only if $G=G_{\frac{n-1}{2}}$ and $s=\frac{n-1}{2}$. 
		It leads to a contradiction with the conditions of Theorem \ref{thm3}.
		This completes the proof.
	\end{proof}

	\begin{theorem}\label{thm4}
		For a positive integer $\delta$ and an even $k$, let $G$ be a graph of even  order $n\geq 4$ with minimum degree $\delta(G) \geq \delta$.
		
		$\mathrm{(1)}$~If $n>6\delta+2$  and $e(G)\geq e(K_\delta \vee (K_{n-2\delta}+\overline{K_\delta}))$, then $G$ is $\text{GBC}_k$ unless $ G = K_\delta \vee (K_{n-2\delta}+\overline{K_\delta})$.
		
		$\mathrm{(2)}$~If $n=6\delta+2$  and $e(G)\geq e(K_\delta \vee (K_{n-2\delta}+\overline{K_\delta}))$, then $G$ is $\text{GBC}_k$ unless $ G = K_\delta \vee (K_{n-2\delta}+\overline{K_\delta})$ or $G =K_{ \frac{n}{2}} \vee \overline{K_{ \frac{n}{2} }} $.
		
		$\mathrm{(3)}$~If $n<6\delta+2$  and $e(G)\geq e(K_{ \frac{n}{2}} \vee \overline{K_{  \frac{n}{2} }})$, then $G$ is $\text{GBC}_k$ unless $ G = K_{ \frac{n}{2} } \vee \overline{K_{ \frac{n}{2} }}$.
	\end{theorem}
	
	\begin{proof}
		By contradiction, assume that $G$ is not $\text{GBC}_k$. By Lemma \ref{lem1}(1), there exists $\varnothing \neq S \subseteq V(G)$ that  satisfies (\ref{3}).
		Let $|S|=s$ and $i(G-S)=i$. If there are nontrivial components in $G-S$, then $G$ is a spanning subgraph of $K_s \vee (K_{n-s-i}+\overline{K_i})$. If $i>s$, we can join the $i-s$ isolated vertices in $G-S$ one by one to all vertices in $V(K_{n-s-i})$. Hence $G\preceq K_s \vee (K_{n-s-i}+\overline{K_i}) \preceq K_s \vee (K_{n-2s}+\overline{K_s})$. Since $\delta(G)\geq \delta$ and $K_{n-2s}$ is a nontrivial component, $\delta \leq s \leq \frac{n-2}{2}$. 	If there is no nontrivial component in $G-S$, $G\preceq K_s \vee \overline{K_{n-s}}\preceq K_{\frac{n}{2}} \vee \overline{K_{\frac{n}{2}}}$. Denote $G_s=K_s \vee (K_{n-2s}+\overline{K_s})$. Then $K_{\frac{n}{2}} \vee \overline{K_{\frac{n}{2}}}=G_{\frac{n}{2}}$. We now compare the size of $G_s$ for $\delta \leq s \leq \frac{n}{2}$.
		
		By Lemma \ref{lem8}, when $n>6\delta+2$, $e(G)\leq e(G_s)\leq e(G_\delta)$  with equality  throughout if and only if $G=G_\delta$ and $s=\delta$. When $n=6\delta+2$, then $e(G)\leq e(G_s)\leq e(G_\delta)=e(G_{\frac{n}{2}})$ with equality throughout if and only if either $G=G_\delta$ and $s=\delta$, or $G=G_{\frac{n}{2}}$ and
		$s=\frac{n}{2}$. When $n<6\delta+2$, $e(G)\leq e(G_s)\leq e(G_{\frac{n}{2}})$ with equality  throughout if and only if $G=G_{\frac{n}{2}}$ and $s=\frac{n}{2}$. 
		It leads to a contradiction with the conditions of Theorem \ref{thm4}.
		This completes the proof.
	\end{proof}

	\begin{theorem}\label{thm5}
		For a positive integer $\delta$ and an even $k$, let $G$ be a graph of odd order $n\geq 8\delta+4$ with minimum degree $\delta(G) \geq \delta$.
		If  $\rho(G)\geq \rho(K_\delta \vee (K_{n-2\delta}+\overline{K_\delta}))$, then $G$ is $\text{GFC}_k$ unless $ G = K_\delta \vee (K_{n-2\delta}+\overline{K_\delta})$.
	\end{theorem}

	\begin{proof}
		By contradiction, assume that $G$ is not $\text{GFC}_k$. 
		By Lemma \ref{lem1}(1), there exists $\varnothing \neq S \subseteq V(G)$ that  satisfies (\ref{3}).
		Let $|S|=s$ and $i(G-S)=i$.
		If there are nontrivial components in $G-S$, then $G$ is a spanning subgraph of $K_s \vee (K_{n-s-i}+\overline{K_i})$. If $i>s$, we can join the $i-s$ isolated vertices in $G-S$ one by one to all vertices in $V(K_{n-s-i})$. Hence $G\preceq K_s \vee (K_{n-s-i}+\overline{K_i}) \preceq K_s \vee (K_{n-2s}+\overline{K_s})$. Since $\delta(G)\geq \delta$ and $K_{n-2s}$ is a nontrivial component, $\delta \leq s \leq \frac{n-2}{2}$. If there is no nontrivial component in $G-S$, $G\preceq K_s \vee \overline{K_{n-s}}\preceq K_{ \frac{n-1}{2}} \vee \overline{K_{\frac{n+1}{2}}}$.
		Denote $G_s=K_s \vee (K_{n-2s}+\overline{K_s})$. Then $K_{\frac{n-1}{2}} \vee \overline{K_{\frac{n+1}{2}}}=G_{\frac{n-1}{2}}$. We now compare the spectral radii of the graphs $G_s$ for $\delta \leq s \leq \frac{n-1}{2}$.
			
		By Lemma \ref{lem9}, if $G$ is not $\text{GFC}_k$, either $G=K_\delta \vee (K_{n-2\delta}+\overline{K_\delta})$ or $\rho(G)< \rho(K_\delta \vee (K_{n-2\delta}+\overline{K_\delta}))$.
		It leads to a contradiction with the conditions of Theorem \ref{thm5}. This completes the proof.
	\end{proof}

	\begin{theorem}\label{thm6}
		For a positive integer $\delta$ and an even $k$, let $G$ be a graph of even order $n\geq 8\delta+4$ with minimum degree $\delta(G) \geq \delta$.
		If  $\rho(G)\geq \rho(K_\delta \vee (K_{n-2\delta}+\overline{K_\delta}))$, then $G$ is $\text{GBC}_k$ unless $ G = K_\delta \vee (K_{n-2\delta}+\overline{K_\delta})$.
	\end{theorem}
	
	\begin{proof}
		
		By contradiction, assume that $G$ is not $\text{GBC}_k$. 
		By Lemma \ref{lem1}(1), there exists $\varnothing \neq S \subseteq V(G)$ that  satisfies (\ref{3}).
		Let $|S|=s$ and $i(G-S)=i$.
		If there are nontrivial components in $G-S$, then $G$ is a spanning subgraph of $K_s \vee (K_{n-s-i}+\overline{K_i})$. If $i>s$, we can join the $i-s$ isolated vertices in $G-S$ one by one to all vertices in $V(K_{n-s-i})$. Hence $G\preceq K_s \vee (K_{n-s-i}+\overline{K_i}) \preceq K_s \vee (K_{n-2s}+\overline{K_s})$. Since $\delta(G)\geq \delta$ and $K_{n-2s}$ is a nontrivial component, $\delta \leq s \leq \frac{n-2}{2}$. If there is no nontrivial component in $G-S$, $G\preceq K_s \vee \overline{K_{n-s}}\preceq K_{ \frac{n}{2}} \vee \overline{K_{\frac{n}{2}}}$.
		Denote $G_s=K_s \vee (K_{n-2s}+\overline{K_s})$. Then $K_{\frac{n}{2}} \vee \overline{K_{\frac{n}{2}}}=G_{\frac{n}{2}}$. We now compare the spectral radii of the graphs $G_s$ for $\delta \leq s \leq \frac{n}{2}$.
		
		By Lemma \ref{lem9}, if $G$ is not $\text{GBC}_k$, either $G=K_\delta \vee (K_{n-2\delta}+\overline{K_\delta})$ or $\rho(G)< \rho(K_\delta \vee (K_{n-2\delta}+\overline{K_\delta}))$.
		It leads to a contradiction with the conditions of Theorem \ref{thm6}. This completes the proof.
	\end{proof}

	\section*{Declaration of competing interests}
	
	The authors declare that they have no known competing financial interests or personal relationships that could have appeared to influence the work reported in this paper.
	
	\section*{Data availability}
	
	No data was used for the research described in the article.
	
	\section*{Acknowledgements}
	
	The authors wish to thank the two anonymous referees for their valuable comments and suggestions, which have greatly helped improve the presentation of this paper.

\end{document}